\documentclass[a4paper,reqno]{amsart}
\usepackage{amssymb,amsmath,amscd} 
\usepackage[usenames,dvipsnames]{color} 
\usepackage{tikz} 
\usepackage[all]{xy}
\usepackage{pstricks}
\usepackage{color}
\usepackage{pdfpages}
\usepackage{caption}
\usepackage{subcaption}
\usepackage{graphicx,wrapfig,lipsum}
\usepackage{mathtools}
\usepackage{epstopdf}

\usetikzlibrary{calc}

\usepackage[left=1.25 in, right=1.25 in]{geometry}

\usetikzlibrary{arrows}
\usetikzlibrary{matrix} 

\newtheorem{theorem}{Theorem}[section]

\newtheorem{lemma}[theorem]{Lemma}
\newtheorem{corollary}[theorem]{Corollary}
\newtheorem*{theorem*}{Theorem}
\theoremstyle{definition}
\newtheorem{definition}[theorem]{Definition}
\newtheorem{remark}[theorem]{Remark}
\newtheorem{example}[theorem]{Example}

\newcommand\restr[2]{{
  \left.\kern-\nulldelimiterspace
  #1
  \vphantom{\big|}
  \right|_{#2}
  }}

\newcommand{\cL}{{\mathcal L}}

\begin{document}

\title[Equivariant classification of  $b^m$-symplectic surfaces]{Equivariant classification of  $b^m$-symplectic surfaces and Nambu structures}

\author{Eva Miranda}\address{ Eva Miranda,
Department of Mathematics, Universitat Polit\`{e}cnica de Catalunya, Barcelona, Spain \it{e-mail: eva.miranda@upc.edu}
 }
\author{Arnau Planas} \address{Arnau Planas, Department of Mathematics, Universitat Polit\`{e}cnica de Catalunya, Barcelona, Spain \it{e-mail: arnauplanasbahi@gmail.com }}
   \thanks{Both authors are supported by the Ministerio de Econom\'{i}a y Competitividad project MTM2015-69135-P/FEDER. Eva Miranda is also supported by the Catalan Grant 2014SGR634. }

\date{\today}

\begin{abstract} In this paper we extend the classification scheme in \cite{scott} for $b^m$-symplectic surfaces and, more generally,  $b^m$-Nambu structures  to the equivariant setting.  When the compact group is the group of deck-transformations of an orientable covering, this yields the classification of these objects for non-orientable surfaces.   The paper also includes recipes to construct $b^m$-symplectic structures on surfaces. Feasibility of such constructions depends on orientability and on the colorability of an associated graph. The desingularization technique in \cite{evajonathanvictor} is revisited for surfaces and the compatibility with this classification scheme is analyzed.
We recast the strategy used in \cite{david} to classify stable Nambu structures of top degree on orientable manifolds to classify $b^m$-Nambu structures (not necessarily oriented) using the language of $b^m$-cohomology. The paper ends up with an equivariant classification theorem of $b^m$-Nambu structures of top degree.

\end{abstract}

\maketitle

\section{Introduction}

The topological classification of compact surfaces is determined by orientability and genus. The geometrical classification of symplectic surfaces  was established by Moser \cite{moser}. Moser proved that any two compact symplectic surfaces with symplectic forms lying on the same De Rham cohomology class  are equivalent in the sense that there exists a diffeomorphism taking one symplectic structure to the other.

Poisson structures show up naturally in this scenario as a generalization of symplectic structures where the non-degeneracy condition is relaxed.
The first examples of Poisson structures are symplectic manifolds and manifolds with the zero Poisson structure. In-between these two extreme examples there is a wide variety of Poisson manifolds. Poisson structures with dense symplectic leaves and controlled singularities have been the object of study of several recent articles (see for instance \cite{guimipi}, \cite{guimipi2}, \cite{gmps}, \cite{evajonathanvictor}, \cite{gualtieri},  \cite{marcutosorno2}). The classification of these objects in dimension $2$ is given by a suitable cohomological condition. In the extreme case of symplectic manifolds this cohomology coincides with de Rham cohomology and, as explained above, this classification was already known to Moser \cite{moser}. For orientable $b$-symplectic manifolds, the classification can be formulated in terms of $b$-cohomology (see \cite{guimipi}) which reinterprets former classification invariants by Radko \cite{radko}.

It is possible to consider other classes of Poisson manifolds with simple singularities like $b^m$-symplectic manifolds \cite{scott} or more general singularities \cite{mirandascott} by relaxing the transversality condition for $b$-symplectic manifolds. These structures have relevance in mechanics: most of the examples are found naturally in the study of celestial mechanics (see  \cite{km}, \cite{dkm}, \cite{kms}).
In the same way,  $b^m$-symplectic structures are classified in terms of $b^m$-cohomology \cite{scott}. The recent papers \cite{cavalcanti}, \cite{marcutosorno2}, \cite{frejlichmartinezmiranda} have renewed interest in the  non-orientable counterparts of these structures.

In this article we focus our attention on two kinds of objects:
  $b^m$-symplectic manifolds and  a generalization of stable Nambu structures which we call $b^m$-Nambu structures of top degree. These objects coincide in dimension 2.
For these we prove Radko-type equivariant classification.
When the group considered is the group of deck-transformations of an orientable covering, this yields the classification of non-orientable compact surfaces in the $b^m$-case. Such a classification was missing in the literature. We also examine the compatibility of this classification scheme and a desingularization procedure described in \cite{evajonathanvictor} proving that, for surfaces, equivalent $b^{2k}$-symplectic structures get mapped to equivalent symplectic structures under the desingularization procedure though non-equivalent $b^{2k}$-symplectic structures might get mapped to equivalent symplectic structures via this procedure.

\textbf{Organization of the paper:} In Section \ref{Preliminaries} we include the necessary preliminaries of $b^m$-structures. In Section \ref{toyexamples} we present some examples of $b^m$-symplectic surfaces on orientable and non-orientable manifolds. In Section \ref{sec:nonorientable} we present an equivariant $b^m$-Moser theorem and use it to classify non-orientable $b^m$-symplectic surfaces. In Section \ref{sec:constructions} we give explicit constructions of $b^m$-symplectic structures  with prescribed critical set depending on orientability and colorability of an associated graph. In Section \ref{sec:desingularization} we analyze the behavior of this classification under the desingularization procedure described in \cite{evajonathanvictor}. In Section \ref{sec:nambu},  $b^m$-Nambu structures are classified using the equivariant techniques.

\section{Preliminaries}\label{Preliminaries}

The context of this paper is the so called $b$-Poisson or $b$-symplectic geometry.
A $b$-Poisson vector field on a manifold $M^{2n}$ is a Poisson vector field such that the map
\begin{equation}\label{eq:transverse}
F: M \rightarrow \bigwedge^{2n} TM: p \mapsto (\Pi(p))^n
\end{equation}
is transverse to the zero section. Then, a pair $(M,\Pi)$ is called a \textbf{$b$-Poisson manifold} and the vanishing set $Z$ of $F$ is called the \textbf{critical hypersurface}.

This class of Poisson structures was studied by Radko \cite{radko} in dimension two and considered in numerous papers in the last years: \cite{guimipi}, \cite{guimipi2}, \cite{gmps}, \cite{evajonathanvictor}, \cite{marcutosorno2}, \cite{gualtieri}, \cite{gualtierietal} and \cite{cavalcanti} among others.

Next, we recall Radko's classification theorem and  the cohomological re-statement presented in \cite{guimipi2}.

\begin{definition}\label{lvol}  The \textbf{Liouville volume} is the following limit: $ V(\Pi )\coloneqq \lim _{\varepsilon \to 0}\int _{|h|>\varepsilon }\omega^n $.
	It exists and is independent of the choice of the defining function $h$ of $Z$ (see \cite{radko} for the proof).
\end{definition}

\begin{definition}
	Let $(M,\Pi)$ be a Poisson manifold, $\Omega$ a volume form on it, and let $u_f$ denote the Hamiltonian vector field of a smooth function $f:M\rightarrow\mathbb{R}$. The \textbf{modular vector field} ($X^{\Omega}$) is the derivation defined as follows:
	$$f\mapsto \frac{\mathcal{L}_{u_f}\Omega}{\Omega}.$$
	Given $\gamma$ a connected component of $Z(\Pi )$ the \textbf{modular period} of $\Pi$ around $\gamma$ is defined by:

	$$T_{\gamma}(\Pi )\coloneqq \textrm{period of }\, {X^{\Omega}}|_{\gamma }. $$
	
\end{definition}

The following theorem classifies $b$-symplectic structures on surfaces using these invariants:
\begin{theorem}[\textbf{Radko} \cite{radko}]\label{Radko}
	Two Poisson structures $\Pi$, $\Pi'$  are globally equivalent if and only if the following invariants coincide:
	\begin{enumerate}
		\item the Liouville volume,
		\item the topology of $Z$ and
		\item the modular periods on each connected component of $Z$.
	\end{enumerate}
\end{theorem}

An appropriate formalism to deal with these structures was introduced in \cite{guimipi} where \textbf{$b$-manifolds}\footnote{The `$b$' of $b$-manifolds stands for `boundary', as initially considered by Melrose for the study of pseudo-differential operators on manifolds with boundary.} are defined to be pairs $(M,Z)$ of a manifold and a hypersurface. In this way the concept of $b$-manifold previously introduced by Melrose is generalized.
A \textbf{$b$-vector field} on a $b$-manifold $(M,Z)$ is a vector field tangent to the hypersurface $Z$ at every point $p\in Z$. Observe that if $x$ is a local defining function for $Z$ then the set of $b$-vector fields is locally generated by

\begin{equation}\label{eq:generatebvectors}
\{x \frac{\partial}{\partial{x}}, \frac{\partial}{\partial{x_1}},\ldots, \frac{\partial}{\partial{x_{n-1}}}\}.
\end{equation}

In contrast with \cite{guimipi}, in this paper we are not requiring the existence of a global defining function for $Z$ and orientability of $M$.
The vector bundle that has as sections the $b$-vector fields is called the \textbf{$b$-tangent bundle} $^b TM$ (its existence is established in \cite{guimipi}).
The \textbf{$b$-cotangent bundle} $^b T^*M$ is defined using duality. A \textbf{$b$-form} is a  section of the $b$-cotangent bundle. Locally these sections are generated by
\begin{equation}\label{eq:generatebforms}
\{\frac{1}{x} dx, d x_1,\ldots, d x_{n-1}\}.
\end{equation}
In the same way we define a \textbf{$b$-form of degree $k$} to be a section of the bundle $\bigwedge^k(^b T^*M)$, the set of these forms is denoted $^b\Omega^k(M)$. Denoting by $f$ the distance function\footnote{Originally in \cite{guimipi} $f$ stands for a global function, but in order to deal with non-orientable manifolds we may proceed as in \cite{marcutosorno2} and use the distance function.} to the set $Z$, we may write the following decomposition as in \cite{guimipi}:

\begin{equation}\label{eq:decomposition}
\omega=\alpha\wedge\frac{df}{f}+\beta, \text{ with } \alpha\in\Omega^{k-1}(M) \text{ and } \beta\in\Omega^k(M).
\end{equation}

This decomposition allows to extend the differential of the De Rham complex $d$ to $^b\Omega(M)$ by setting $d\omega=d\alpha\wedge\frac{df}{f}+d\beta.$ The cohomology associated to this complex is called \textbf{$b$-cohomology} and it is denoted by \textbf{$^b H^*(M)$}.

Once reached this point a \textbf{$b$-symplectic} form on a $b$-manifold $(M^{2n},Z)$ is defined as a non-degenerate closed $b$-form of degree $2$ (i.e., $\omega_p$ is of maximal rank as an element of $\Lambda^2(\,^b T_p^* M)$ for all $p\in M$). The notion to $b$-symplectic forms is dual to notion of $b$-Poisson structures. The advantage of using forms is that symplectic tools can be `easily' exported.

Radko's classification theorem can be translated  into this language, as was done in \cite{guimipi}: 

 \begin{theorem}[\textbf{Radko's theorem in $b$-cohomological language, \cite{guimipi2}}] Let $S$ be a  compact orientable surface and  and let  $\omega_0$ and $\omega_1$ be two $b$-symplectic forms on $(M,Z)$  defining the same $b$-cohomology class (i.e.,$[\omega_0]= [\omega_1]$).  Then there exists a diffeomorphism $\phi:M\rightarrow M$ such that $\phi^*\omega_1 = \omega_0$.
\end{theorem}

By relaxing the transversality condition (for instance, asking that $\Pi^n$ has a singularity of type $A_n$ on the Arnold's list, see \cite{arnold} and \cite{arnold2}), Scott \cite{scott} defined the $b^m$-Poisson structures.

Proceeding \emph{mutatis mutandis} as in the $b$-case one defines the $b^m$-tangent bundle ($^{b^m} TM $), the  $b^m$-cotangent bundle ($^{b^m} T^*M $), the $b^m$-De Rham complex and the $b^m$-symplectic structures. For instance, the definition of the $b^m$-tangent bundle can be defined as the bundle whose sections are generated by:

\begin{equation}\label{eq:generatebmvectors}
\{x^m \frac{\partial}{\partial{x}}, \frac{\partial}{\partial{x_1}},\ldots, \frac{\partial}{\partial{x_{n-1}}}\},
\end{equation}

with $x$ such that $|x| = \lambda$, and $\lambda$ is the distance function to $z$.

A {\bf Laurent Series} of a closed $b^m$-form $\omega$ is a decomposition of $\omega$ in a tubular neighborhood $U$ of $Z$
of the form
\begin{equation}\label{eqn:laurent}
\omega = \frac{dx}{x^m} \wedge (\sum_{i = 0}^{m-1}\pi^*(\alpha_{i})x^i) + \beta
\end{equation}

\noindent with $\pi: U \to Z$ the projection of the tubular neighborhood onto $Z$, $\alpha_{i}$ a closed smooth De Rham form on $Z$ and $\beta$  a De Rham form on $M$.

In \cite{scott} it is proved that in  a  neighborhood of $Z$, every   closed $b^m$-form  $\omega$ can be written in a  Laurent form of type (\ref{eqn:laurent}) having fixed a (semi)local defining function.

 $b^m$-Cohomology is related to de Rham cohomology via the following theorem:
\begin{theorem}[\textbf{$b^m$-Mazzeo-Melrose}, \cite{scott}]\label{thm:Mazzeo-Melrose}
\begin{equation}\label{eqn:Mazzeo-Melrose}
 ^{b^m}H^p(M) \cong H^p(M)\oplus(H^{p-1}(Z))^m.
\end{equation}
\end{theorem}

The Moser path method can be generalized to $b^m$-symplectic structures:

\begin{theorem}[\textbf{Moser's path method}, \cite{evajonathanvictor}] Let $\omega_0, \omega_1$ be two $b^m$-symplectic forms on $(M^{2n}, Z)$ with $Z$ compact and
$\omega_0\vert_{{Z}} = \omega_1\vert_{Z}$, then there are neighborhoods $U_0, U_1$ of $Z$ and a $b^m$-symplectomorphism $\varphi: (U_0, Z,
\omega_0) \rightarrow (U_1,Z,  \omega_1)$ such that $\varphi\vert_{Z} = Id.$ \end{theorem}

One of the results that follows from the Moser's path method is a local description of a $b^m$-symplectic manifold.

\begin{theorem}[\textbf{$b^m$-Darboux theorem, \cite{evajonathanvictor}}]\label{theorem:Darbouxbn}
Let $\omega$ be a $b^m$-symplectic form on $(M,Z)$ and $p\in Z$. Then we can find a coordinate chart $(U,x_1,y_1,\ldots,x_n,y_n)$ centered at
$p$ such that on $U$ the hypersurface $Z$ is locally defined by $x_1=0$ and
$$\omega=\frac{d x_1}{x_1^m}\wedge {d y_1}+\sum_{i=2}^n d x_i\wedge d y_i.$$
\end{theorem}

Another consequence of Moser's path method is a global classification of $b^m$-symplectic surfaces \`{a} la Radko in terms of $b^m$-cohomology classes.

\begin{theorem}[\textbf{Classification of orientable $b^m$-manifolds, \cite{scott}}]\label{thm:scott}
Let $\omega_0$ and $\omega_1$ be two $b^m$-symplectic forms on a compact connected $b^m$-surface $(S,Z)$.
Then, the following conditions are equivalent:
\begin{enumerate}
  \item their $b^m$-cohomology classes coincide $[\omega_0] = [\omega_1]$,
  \item the surfaces are globally $b^m$-symplectomorphic,
  \item the Liouville volumes of $\omega_0$ and $\omega_1$ and the numbers $$\int_{\gamma} \alpha_{i}$$ for all connected
  components $\gamma \subseteq Z$ and all $1 \leq i \leq m$ coincide (where
  $\alpha_{i}$ are the one-forms appearing in the Laurent decomposition
  of the two 2-$b^m$-forms $\omega_0$ and $\omega_1$).
\end{enumerate}
\end{theorem}

\section{ Toy examples of $b^m$-symplectic surfaces }\label{toyexamples}

In this section we describe some examples of orientable and non-orientable $b^m$-symplectic surfaces.
\begin{enumerate}
\item \textbf{A $b^{m}$-symplectic structure on the sphere:}\label{ex:sphere}
Consider the sphere $S^2$ with the $b^m$-symplectic form $\omega = \frac{1}{h^m}dh\wedge d\theta$, where $h$ stands for the height and $\theta$ for the angular coordinate. Observe that this form has the equator as the critical set $Z$.
\item \textbf{A $b^{m}$-symplectic structure on the torus:}\label{ex:torus} Consider $\mathbb{T}^2$ as quotient of the plane ($\mathbb{T}^2 = \{(x, y) \in (\mathbb{R}/\mathbb Z)^2\}$).
Let $\omega = \frac{1}{(\sin 2\pi y)^n}dx\wedge dy$ be a $b^m$-symplectic structure on $\mathbb R^2$. The action of $\mathbb Z^2$ leaves this form invariant and therefore this $b^m$-form descends to the quotient. Observe that this $b^m$-form defines $Z = \{y \in \{0,\frac{1}{2}\}\}$.

\begin{figure}[h!]
    \centering

    \begin{tikzpicture}[scale = 0.35]
\draw [rotate = 0,black, line width=1pt] (-7,0) ellipse (5cm and 3cm);
\draw [black, line width=1pt] plot [smooth, tension=1.3] coordinates { (-9,.5)(-7,-.5)(-5,.5) };
\draw [black, line width=1pt] plot [smooth, tension=1.3] coordinates { (-8.5,0)(-7,.5)(-5.5,0) };

\draw [magenta, line width=1pt] plot [smooth, tension=1.3] coordinates { (-7,-.5)(-7.5,-1.75)(-7,-3)};
\draw [magenta, dashed, line width=1pt] plot [smooth, tension=1.3] coordinates { (-7,-.5)(-6.5,-1.75)(-7,-3)};
\draw [magenta, line width=1pt] plot [smooth, tension=1.3] coordinates { (-7,.5)(-7.5,1.75)(-7,3)};
\draw [magenta, dashed, line width=1pt] plot [smooth, tension=1.3] coordinates { (-7,.5)(-6.5,1.75)(-7,3)};

\end{tikzpicture}

    \caption{Example: $b^{m}$-symplectic structure in the torus.}
    \label{fig:torus}
\end{figure}
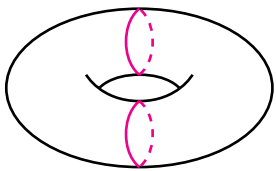

\item \textbf {A $b^{2k+1}$-symplectic structure on the projective space:}\label{example:projective}  Consider the previous example of $b^m$-symplectic form on the sphere and consider the antipodal action on it.  Observe that $\omega$ is invariant by the action for $m=2k+1$, yielding  a $b^{2k+1}$-symplectic form in $\mathbb{RP}^2$
with critical set $Z$, that is  the equator modulo the antipodal identification (thus diffeomorphic to $S^1$). Observe that a neighborhood of $Z$ is diffeomorphic to the Moebius band.

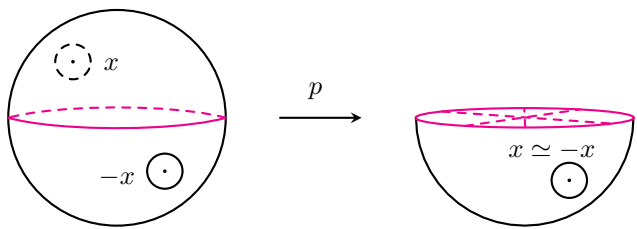
\begin{figure}[h!]
\centering

\begin{tikzpicture}
[thick, scale=0.55, line cap=round,line join=round,>=triangle 45,x=.65cm,y=.65cm]
\clip(-6,-4.1) rectangle (22,4.1);
\draw[dashed](-1.62,2.08) circle (0.6511528238439883 );
\draw(1.76,-1.98) circle (0.6511528238439883 );
\draw(16.64,-2.32) circle (0.6511528238439883 );
\draw(0.,0.) circle (4. );
\draw [->, >=stealth] (6.,0.) -- (9.,0.);
\draw [shift={(15.,0.)}] plot[domain=3.141592653589793:6.283185307179586,variable=\t]({1.*4.*cos(\t r)+0.*4.*sin(\t r)},{0.*4.*cos(\t r)+1.*4.*sin(\t r)});
\draw [color=magenta][rotate around={0.:(14.999999999999828,0.)}] (14.999999999999828,0.) ellipse (4.016167727329326  and 0.3600044639249609 );
\draw [dashed, color=magenta](12.799930602626697,-0.3011823446063568)-- (17.040532763637565,0.31007549425931386);
\draw [dashed, color=magenta](15.019946575672243,-0.36000002381646556)-- (14.98,0.36);
\draw [dashed, color=magenta](11.944866403555565,0.2336765444470393)-- (18.221799278450536,-0.21493905572999872);
\draw[color=black] (7.3,1) node {$p$};
\draw [fill=black] (-1.62,2.08) circle (0.5pt);
\draw[color=black] (-0.2,2) node {$x$};
\draw [fill=black] (1.76,-1.98) circle (0.5pt);
\draw[color=black] (0,-2.2) node {$-x$};
\draw [fill=black] (16.64,-2.32) circle (0.5pt);
\draw[color=black] (16,-1.2) node {$x \simeq -x$};

  \coordinate (P) at ($(0, 0) + (30:3cm and 2cm)$);
  \draw[dashed,color=magenta] ($(0, 0) + (30:3cm and 0cm)$(P) arc
  (30:150:3cm and 0.5cm);
  \draw[color=magenta] ($(0, 0) + (30:3cm and 0cm)$(P) arc
  (30:150:3cm and -0.5cm);

\end{tikzpicture}

\caption{The $b^{2k+1}$-symplectic structure on the sphere $S^2$ that vanishes at the equator induces a $b^{2k+1}$-symplectic structure on the projective space $\mathbb{RP}^2$.}
\label{fig:my_label}
\end{figure}

\item \textbf{ A $b^{2k+1}$-symplectic structure on a Klein bottle:} 
Consider the Torus with the structure given in the previous example.


 Consider the group action $\mathbb{Z}/2\mathbb{Z}$ that acts on $(x,y) \in \mathbb{T}^2$ by $\text{Id}\cdot(x,y) = (x,y)$ and
$\text{-Id}\cdot(x,y) =(1-x,y)$. The orbit space by this action is the Klein bottle.

Thus, the $b^m$-symplectic form $\omega = \frac{dx}{(\sin 2\pi x)^m}\wedge dy$ induces a $b^m$-symplectic structure in $T$ if $\omega$ is invariant by the action of the group. Let $\rho_{-\text{id}}$ denote morphism induced by the action of ${-\text{id}}$.

\begin{tabular}{rcl}
$\displaystyle\rho_{-\text{id}}^{*}\omega$& $=$  & $\displaystyle\rho_{-\text{id}}^{*}\left(\frac{dx}{(\sin 2\pi x)^m}\wedge y\right) = \displaystyle \frac{d(1-x)}{(\sin(2\pi- 2\pi x))^m}\wedge dy$\\
 & $=$ & $\displaystyle\frac{-dx}{(-1)^m\sin 2\pi x} \wedge dy.$\\
\end{tabular}

Thus, $\omega$ is invariant if and only if $m$ is odd, and in this case we have constructed an example of $b^m$-symplectic structure in the Klein bottle.
\end{enumerate}

\begin{remark}
The previous examples only exhibit  $b^{2k+1}$-symplectic structures on non-orientable surfaces. As we will see in Section \ref{sec:constructions} only orientable surfaces can admit $b^{2k}$-symplectic structures.
\end{remark}
\section{Equivariant classification of $b^m$-surfaces. Non-orientable $b^m$-surfaces.}\label{sec:nonorientable}

In this section we give an equivariant Moser theorem for $b^m$-symplectic manifolds. This yields the classification of non-orientable surfaces thus extending the classification theorems of Radko and Scott for non-orientable surfaces.

 A classification of $b$-symplectic surfaces \`{a} la Moser was obtained in  \cite{guimipi} in terms of the class  determined by the $b$-symplectic structures in $b$-cohomology. This result resembles the one of Moser for the global  classification of symplectic surfaces using De Rham cohomology classes. The class in $b$-cohomology determines, in its turn, the period of the modular vector field along the connected components (curves) of the singular locus.

We now extend the classification result for manifolds admitting a group action leaving the $b^m$-symplectic structure invariant.

\begin{theorem}[\textbf{Equivariant $b^m$-Moser theorem for surfaces}]\label{emt} Suppose that $S$ is a compact surface, let $Z$ be a union of non-intersecting curves and let  $\omega_0$ and $\omega_1$ be two $b^m$-symplectic structures on $(S,Z)$ which are invariant under the action of a compact Lie group $\rho:G\times S\longrightarrow S$ and defining the same $b^m$-cohomology class,  $[\omega_0]=[\omega_1]$. Then, there exists an equivariant diffeomorphism $\gamma_1:S\rightarrow S$,  such that $\gamma_1$ leaves $Z$ invariant and satisfies $\gamma_1^*\omega_1 = \omega_0$.
\end{theorem}

\begin{proof} Denote by $\rho: G\times S\longrightarrow S$ the group action and  denote by $\rho_g$ the induced diffeomorphism for a fixed $g\in G$. i.e., $\rho_g(x):=\rho(g,x)$.
 Consider the linear family of $b^m$-forms $\omega_s=s \omega_1+ (1-s)\omega_0$. Since the manifold is a surface, the fact that $\omega_0$ and $\omega_1$ are non-degenerate $b^m$-forms and of the same sign on $S\setminus Z$  (thus non-vanishing sections of $\Lambda^2 (^b T^*(S))$) implies that the linear path is non-degenerate too.
We will prove that there exists a family
$\gamma_s:S\rightarrow S$, with $0\leq s\leq1$ such that

 \begin{equation}\label{gammat}\gamma_s^*\omega_s=\omega_0.\end{equation}
This is equivalent to the following equation,\begin{equation}\label{diffagammat}\cL_{X_s}\omega_s=\omega_0-\omega_1,\end{equation}
where $X_s=\frac{d\gamma_s}{ds}\circ\gamma_s^{-1}$.

 Since the cohomology class of both forms coincide, $\omega_1-\omega_0= d\alpha$ for $\alpha$ a $b^m$-form of degree $1$.

 Therefore equation (\ref{diffagammat}) becomes \begin{equation}\label{vt}\iota_{X_s}\omega_s=-\alpha.\end{equation}

This equation has a unique solution $X_s$ because  $\omega_s$ is $b^m$-symplectic and therefore it is non-degenerate. Furthermore, the solution is a $b^m$-vector field. From this solution we will construct an equivariant solution such that its $s$-dependent flow gives an equivariant diffeomorphism.

Since the forms $\omega_0$ and $\omega_1$ are $G$-invariant, we can find a $G$-invariant primitive $\tilde\alpha$ by averaging with respect to a Haar measure the initial form  $\alpha$: $\tilde\alpha=\int_G \rho_g^*(\alpha) d\mu$ and therefore
the invariant vector field,

$$X_s^G=\int_G \rho_{g_*}(X_s) d\mu$$ is a solution of the equation,

\begin{equation}\iota_{X_s^G}\omega_s=-\tilde{\alpha}.
\end{equation}
We can get an equivariant $\gamma_s^G$ by integrating $X_s^G$. This family satisfies $\gamma_t^{G*}\omega_t = \omega_0$ and it is equivariant.

Also observe that since $X_t^G$ is tangent to $Z$, this diffeomorphism preserves $Z$.

\end{proof}

A non-orientable manifold can be seen as a pair $(\tilde M, \rho)$ with $\tilde M$ the orientable covering and $\rho$  the action given by deck-transformations of $\mathbb Z/2\mathbb Z$ on $\tilde M$. This perspective is very convenient for classification issues because equivariant mappings on the orientable covering yield actual diffeomorphism on the non-orientable manifolds. We adopt this point of view to provide a classification theorem for non-orientable $b^m$-surfaces in cohomological terms.

\begin{corollary}\label{non-orientable}
Let $S$ be a non-orientable compact surface and let $\omega_1$ and $\omega_2$ be two $b^m$-symplectic forms on $S$. Assume $[\omega_1] = [\omega_2]$ in $b^m$-cohomology then  $(S,\omega_1)$ is equivalent to $(S,\omega_2)$, i.e., there exists a diffeomorphism $\varphi:S \rightarrow S$ such that $\varphi^{*}\omega_2 = \omega_1$.
\end{corollary}

\begin{proof}

Let $p:\tilde{S}\rightarrow S$ be a covering map, and $\tilde{S}$ the orientation double cover.
Since $[\omega_1] = [\omega_2]$  then $[p^*(\omega_1)] = [p^*(\omega_2)]$. Because
$\tilde{S}$ is orientable we may use the $b^m$-Moser theorem  in order to guarantee
the existence of a symplectomorphism $\tilde{\varphi}:(\tilde{S}, p^{*}(\omega_1))
\rightarrow (\tilde{S}, p^{*}(\omega_2))$.
Consider the following diagram:

\begin{equation}\label{CommutativeDiagram}
\begin{tikzpicture}
  \matrix (m) [matrix of math nodes,row sep=3em,column sep=4em,minimum width=2em]
  {
     (\tilde{S}, p^{*}(\omega_1)) & (\tilde{S}, p^{*}(\omega_2)) \\
     (S, \omega_1) & (S,\omega_2) \\};
  \path[-]
    (m-1-1) edge[->] node [left] {$p$} (m-2-1)
            edge[->] node [below] {$\tilde{\varphi}$} (m-1-2)
    (m-1-2) edge[->] node [right] {$p$} (m-2-2);

    \path[dashed]
    (m-2-1.east|-m-2-2) edge[->] node [below] {$\varphi$}
    (m-2-2);

\end{tikzpicture}
\end{equation}


In order that there exists a symplectomorphism $\varphi$ that make the diagram commute above commute we apply the universal property of the quotient to $p \circ \tilde{\varphi}$: there exists a unique $\varphi$ making the diagram commute if and only if the images by $p \circ \tilde{\varphi}$ of the identified points coincide. This is equivalent to requiring that the images of $p$ by $\tilde{\varphi}$ are sent to the orbit of $p$ and this is true because $\tilde{\varphi}(gp) = g\tilde{\varphi}(p)$ as a consequence of  Theorem \ref{emt} (for $G = \mathbb{Z}/2\mathbb{Z}$). It is possible to apply this theorem since the symplectomorphism between $p^{*}\omega_1$ and $p^{*}\omega_2$ given by $b^m$-Moser theorem, yields a family of forms with invariant $b^m$-cohomology class.
\end{proof}

A similar equivariant $b^m$-Moser theorem as theorem \ref{emt} holds for higher dimensions. In that case we need to require that there exists a path $\omega_t$ of $b^m$-symplectic structures connecting $\omega_0$ and $\omega_1$, which is not true in general \cite{mcduff}. The proof follows the same lines  as Theorem \ref{emt}. Such a result was already proved for $b$-symplectic manifolds (see Theorem 8 in \cite{gmps}).

\begin{theorem}[\textbf{Equivariant $b^m$-Moser theorem}]\label{emt2} Let $M$ be a compact manifold and let $Z$ be a smooth hypersurface. Consider $\omega_t$ for $0\leq t \leq 1$,  a smooth family of $b^m$-symplectic forms on $(M,Z)$  such that the $b^m$-cohomology class $[\omega_t]$ does not depend on $t$.

 Assume that the family of $b^m$-symplectic structures is invariant by the action of a compact Lie group $G$ on $M$, then, there exists a family of  equivariant diffeomorphisms $\phi_t:M\rightarrow M$, with $0\leq t \leq 1$ such that $\phi_t$ leaves $Z$ invariant and satisfies $\phi_t^*\omega_t = \omega_0$.
\end{theorem}

\section{Constructions and classification of $b^m$-symplectic structures}\label{sec:constructions}

In this section we describe constructions of $b^m$-symplectic structures on surfaces. We start associating a graph to a $b$-manifold $(M,Z)$.

\subsection{$b$-Graphs}

Let us start this section with some definitions.

\begin{definition}\label{def:graph}
We define the associated graph to a $b$-manifold $(M,Z)$, as the graph with set of vertices given by the connected components of $S\setminus Z$ and with edges connecting two vertices when the connected components associated to them $U_1$ and $U_2$, $\partial U_1 \cap  \partial U_1 \neq \{0\}$.
\end{definition}

The associated graph to a $b$-manifold can contain loops and double edges (that is an edge can connect the same vertex). See Figure \ref{fig:graphs}. 

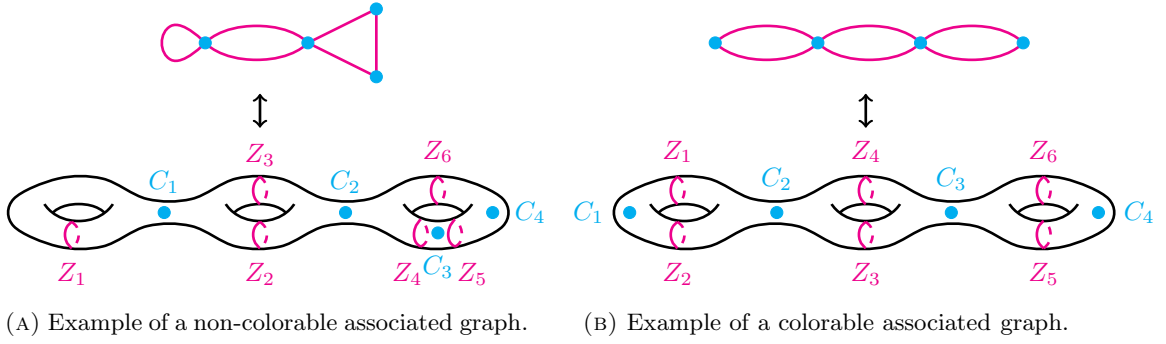
\begin{figure}
\begin{center}
\begin{subfigure}[t]{.47\textwidth}
    \begin{tikzpicture}[scale = 0.45]
\newcommand \unamica{0.4}
\newcommand \factorpetit{1.3}
\newcommand \factorpetitdos{0.5}
\newcommand \transun{0.7}
\newcommand \transdos{-4.5}
\newcommand \transtres{-9.8}
\newcommand \dreta{0.9}
\draw [black,line width=1pt] plot [smooth cycle, tension=0.6] coordinates {
(-9*\factorpetit,0 - \unamica)
(-9*\factorpetit,0 + \unamica)
(-8*\factorpetit,1)
(-7*\factorpetit,1)
(-6*\factorpetit,0 + \unamica)
(-5*\factorpetit,0 + \unamica)
(-4*\factorpetit,1)
(-3*\factorpetit,1)
(-2*\factorpetit,0 + \unamica)
(-1*\factorpetit,0 + \unamica)
(0*\factorpetit,1)
(1*\factorpetit,1)
(2*\factorpetit,0 + \unamica)
(2*\factorpetit,0 - \unamica)
(1*\factorpetit,-1)
(0*\factorpetit,-1)
(-1*\factorpetit,0 - \unamica)
(-2*\factorpetit,0- \unamica)
(-3*\factorpetit,-1)
(-4*\factorpetit,-1)
(-5*\factorpetit,0- \unamica)
(-6*\factorpetit,0- \unamica)
(-7*\factorpetit,-1)
(-8*\factorpetit,-1)};

\draw [black, line width=1pt] plot [smooth, tension=1.3] coordinates { (-2*\factorpetitdos +\transun,.5*\factorpetitdos)(0*\factorpetitdos + \transun,-.5*\factorpetitdos)(2*\factorpetitdos+ \transun,.5*\factorpetitdos) };
\draw [black, line width=1pt] plot [smooth, tension=1.3] coordinates { (-1.5*\factorpetitdos+ \transun,0*\factorpetitdos)(0*\factorpetitdos+ \transun,.5*\factorpetitdos)(1.5*\factorpetitdos+ \transun,0*\factorpetitdos) };

\draw [black, line width=1pt] plot [smooth, tension=1.3] coordinates { (-2*\factorpetitdos +\transdos,.5*\factorpetitdos)(0*\factorpetitdos + \transdos,-.5*\factorpetitdos)(2*\factorpetitdos+ \transdos,.5*\factorpetitdos) };
\draw [black, line width=1pt] plot [smooth, tension=1.3] coordinates { (-1.5*\factorpetitdos+ \transdos,0*\factorpetitdos)(0*\factorpetitdos+ \transdos,.5*\factorpetitdos)(1.5*\factorpetitdos+ \transdos,0*\factorpetitdos) };

\draw [black, line width=1pt] plot [smooth, tension=1.3] coordinates { (-2*\factorpetitdos +\transtres,.5*\factorpetitdos)(0*\factorpetitdos + \transtres,-.5*\factorpetitdos)(2*\factorpetitdos+ \transtres,.5*\factorpetitdos) };
\draw [black, line width=1pt] plot [smooth, tension=1.3] coordinates { (-1.5*\factorpetitdos+ \transtres,0*\factorpetitdos)(0*\factorpetitdos+ \transtres,.5*\factorpetitdos)(1.5*\factorpetitdos+ \transtres,0*\factorpetitdos) };

\draw [magenta, line width=1pt] plot [smooth, tension=1.3] coordinates { (-10,-.28)(-10.2,-.65)(-10,-1.05)};
\draw [magenta, dashed, line width=1pt] plot [smooth, tension=1.3] coordinates { (-10,-.28)(-9.8,-.65)(-10,-1.05)};

\draw [magenta, line width=1pt] plot [smooth, tension=1.3] coordinates { (-4.5,-.28)(-4.7,-.65)(-4.5,-1.05)};
\draw [magenta, dashed, line width=1pt] plot [smooth, tension=1.3] coordinates { (-4.5,-.28)(-4.3,-.65)(-4.5,-1.05)};

\draw [magenta, line width=1pt] plot [smooth, tension=1.3] coordinates { (-4.5,1.05)(-4.7,.65)(-4.5,.25)};
\draw [magenta, dashed, line width=1pt] plot [smooth, tension=1.3] coordinates { (-4.5,1.05)(-4.3,.65)(-4.5,.25)};

\draw [magenta, line width=1pt] plot [smooth, tension=1.3] coordinates { (.7,1.05)(.5,.65)(.7,.25)};
\draw [magenta, dashed, line width=1pt] plot [smooth, tension=1.3] coordinates { (.7,1.05)(.9,.65)(.7,.25)};

\draw [magenta, line width=1pt] plot [smooth, tension=1.3] coordinates { (.2,-1.0)(.0,-.65)(.2,-.2)};
\draw [magenta, dashed, line width=1pt] plot [smooth, tension=1.3] coordinates { (.2,-1.0)(.4,-.65)(.2,-.2)};

\draw [magenta, line width=1pt] plot [smooth, tension=1.3] coordinates { (1.2,-1.0)(1.0,-.65)(1.2,-.2)};
\draw [magenta, dashed, line width=1pt] plot [smooth, tension=1.3] coordinates { (1.2,-1.0)(1.4,-.65)(1.2,-.2)};

\draw[color=magenta] (-10,-1.5 - 0.3) node {$Z_1$};
\draw[color=magenta] (-4.5,-1.5- 0.3) node {$Z_2$};
\draw[color=magenta] (-4.5,1.5+ 0.2) node {$Z_3$};
\draw[color=magenta] (.1-0.3,-1.5- 0.3) node {$Z_4$};
\draw[color=magenta] (1.3+0.4,-1.5- 0.3) node {$Z_5$};
\draw[color=magenta] (.7,1.5+ 0.3) node {$Z_6$};

\draw [color= cyan, fill=cyan] (2.3,0) circle (5pt);
\draw[color=cyan] (3.2+.2,0) node {$C_4$};

\draw [color= cyan, fill=cyan] (-2,0) circle (5pt);
\draw[color=cyan] (-2,.7+.2) node {$C_2$};

\draw [color= cyan, fill=cyan] (.7,-.6) circle (5pt);
\draw[color=cyan] (.7,-1.5-0.1) node {$C_3$};

\draw [color= cyan, fill=cyan] (-7.3,0) circle (5pt);
\draw[color=cyan] (-7.3,.7+.2) node {$C_1$};

\draw [<->, black, line width=1pt] plot [smooth, tension=1.3] coordinates { (-4.5,2.5)(-4.5,3.5)};

\draw [magenta, line width=1pt] plot [smooth, tension=1.3] coordinates { (-4+\dreta,5)(-5.5+\dreta,5.5)(-7+\dreta,5)};
\draw [magenta, line width=1pt] plot [smooth, tension=1.3] coordinates { (-4+\dreta,5)(-5.5+\dreta,4.5)(-7+\dreta,5)};
\draw [magenta, line width=1pt] plot [smooth, tension=1.3] coordinates { (-7+\dreta,5)(-8+\dreta,4.5)(-8+\dreta,5.5)(-7+\dreta,5)};
\draw [magenta, line width=1pt] plot [smooth, tension=1.3] coordinates { (-4+\dreta,5)(-2+\dreta,6)};
\draw [magenta, line width=1pt] plot [smooth, tension=1.3] coordinates { (-4+\dreta,5)(-2+\dreta,4)};
\draw [magenta, line width=1pt] plot [smooth, tension=1.3] coordinates { (-2+\dreta,6)(-2+\dreta,4)};

\draw [color= cyan, fill=cyan] (-4 + \dreta,5) circle (5pt);
\draw [color= cyan, fill=cyan] (-7 + \dreta,5) circle (5pt);
\draw [color= cyan, fill=cyan] (-2 + \dreta,6) circle (5pt);
\draw [color= cyan, fill=cyan] (-2 + \dreta,4) circle (5pt);

\end{tikzpicture}
\caption{Example of a non-colorable associated graph.}
\label{fig:graph}
\end{subfigure}
\quad
\begin{subfigure}[t]{.47\textwidth}
    \begin{tikzpicture}[scale = 0.45]
\newcommand \unamica{0.4}
\newcommand \factorpetit{1.3}
\newcommand \factorpetitdos{0.5}
\newcommand \transun{0.7}
\newcommand \transdos{-4.5}
\newcommand \transtres{-9.8}
\newcommand \dreta{0.1}
\draw [black,line width=1pt] plot [smooth cycle, tension=0.6] coordinates {
(-9*\factorpetit,0 - \unamica)
(-9*\factorpetit,0 + \unamica)
(-8*\factorpetit,1)
(-7*\factorpetit,1)
(-6*\factorpetit,0 + \unamica)
(-5*\factorpetit,0 + \unamica)
(-4*\factorpetit,1)
(-3*\factorpetit,1)
(-2*\factorpetit,0 + \unamica)
(-1*\factorpetit,0 + \unamica)
(0*\factorpetit,1)
(1*\factorpetit,1)
(2*\factorpetit,0 + \unamica)
(2*\factorpetit,0 - \unamica)
(1*\factorpetit,-1)
(0*\factorpetit,-1)
(-1*\factorpetit,0 - \unamica)
(-2*\factorpetit,0- \unamica)
(-3*\factorpetit,-1)
(-4*\factorpetit,-1)
(-5*\factorpetit,0- \unamica)
(-6*\factorpetit,0- \unamica)
(-7*\factorpetit,-1)
(-8*\factorpetit,-1)};

\draw [black, line width=1pt] plot [smooth, tension=1.3] coordinates { (-2*\factorpetitdos +\transun,.5*\factorpetitdos)(0*\factorpetitdos + \transun,-.5*\factorpetitdos)(2*\factorpetitdos+ \transun,.5*\factorpetitdos) };
\draw [black, line width=1pt] plot [smooth, tension=1.3] coordinates { (-1.5*\factorpetitdos+ \transun,0*\factorpetitdos)(0*\factorpetitdos+ \transun,.5*\factorpetitdos)(1.5*\factorpetitdos+ \transun,0*\factorpetitdos) };

\draw [black, line width=1pt] plot [smooth, tension=1.3] coordinates { (-2*\factorpetitdos +\transdos,.5*\factorpetitdos)(0*\factorpetitdos + \transdos,-.5*\factorpetitdos)(2*\factorpetitdos+ \transdos,.5*\factorpetitdos) };
\draw [black, line width=1pt] plot [smooth, tension=1.3] coordinates { (-1.5*\factorpetitdos+ \transdos,0*\factorpetitdos)(0*\factorpetitdos+ \transdos,.5*\factorpetitdos)(1.5*\factorpetitdos+ \transdos,0*\factorpetitdos) };

\draw [black, line width=1pt] plot [smooth, tension=1.3] coordinates { (-2*\factorpetitdos +\transtres,.5*\factorpetitdos)(0*\factorpetitdos + \transtres,-.5*\factorpetitdos)(2*\factorpetitdos+ \transtres,.5*\factorpetitdos) };
\draw [black, line width=1pt] plot [smooth, tension=1.3] coordinates { (-1.5*\factorpetitdos+ \transtres,0*\factorpetitdos)(0*\factorpetitdos+ \transtres,.5*\factorpetitdos)(1.5*\factorpetitdos+ \transtres,0*\factorpetitdos) };

\draw [magenta, line width=1pt] plot [smooth, tension=1.3] coordinates { (-10,-.28)(-10.2,-.65)(-10,-1.05)};
\draw [magenta, dashed, line width=1pt] plot [smooth, tension=1.3] coordinates { (-10,-.28)(-9.8,-.65)(-10,-1.05)};

\draw [magenta, line width=1pt] plot [smooth, tension=1.3] coordinates { (-10,1.05)(-10.2,.65)(-10,.25)};
\draw [magenta, dashed, line width=1pt] plot [smooth, tension=1.3] coordinates { (-10,1.05)(-9.8,.65)(-10,.25)};

\draw [magenta, line width=1pt] plot [smooth, tension=1.3] coordinates { (-4.5,-.28)(-4.7,-.65)(-4.5,-1.05)};
\draw [magenta, dashed, line width=1pt] plot [smooth, tension=1.3] coordinates { (-4.5,-.28)(-4.3,-.65)(-4.5,-1.05)};

\draw [magenta, line width=1pt] plot [smooth, tension=1.3] coordinates { (-4.5,1.05)(-4.7,.65)(-4.5,.25)};
\draw [magenta, dashed, line width=1pt] plot [smooth, tension=1.3] coordinates { (-4.5,1.05)(-4.3,.65)(-4.5,.25)};

\draw [magenta, line width=1pt] plot [smooth, tension=1.3] coordinates { (.7,1.05)(.5,.65)(.7,.25)};
\draw [magenta, dashed, line width=1pt] plot [smooth, tension=1.3] coordinates { (.7,1.05)(.9,.65)(.7,.25)};

\draw [magenta, line width=1pt] plot [smooth, tension=1.3] coordinates { (.7,-1.05)(.5,-.65)(.7,-.25)};
\draw [magenta, dashed, line width=1pt] plot [smooth, tension=1.3] coordinates { (.7,-1.05)(.9,-.65)(.7,-.25)};

\draw[color=magenta] (-10,1.5+.3) node {$Z_1$};
\draw[color=magenta] (-10,-1.5-.3) node {$Z_2$};
\draw[color=magenta] (-4.5,-1.5-.3) node {$Z_3$};
\draw[color=magenta] (-4.5,1.5+.3) node {$Z_4$};
\draw[color=magenta] (.7,-1.5-.3) node {$Z_5$};
\draw[color=magenta] (.7,1.5+.3) node {$Z_6$};

\draw [color= cyan, fill=cyan] (2.3,0) circle (5pt);
\draw[color=cyan] (3.2+.3,0) node {$C_4$};

\draw [color= cyan, fill=cyan] (-2,0) circle (5pt);
\draw[color=cyan] (-2,.7+.3) node {$C_3$};

\draw [color= cyan, fill=cyan] (-7.1,0) circle (5pt);
\draw[color=cyan] (-7.1,.7+.3) node {$C_2$};

\draw [color= cyan, fill=cyan] (-11.4,0) circle (5pt);
\draw[color=cyan] (-12.3-.3,0) node {$C_1$};

\draw [<->, black, line width=1pt] plot [smooth, tension=1.3] coordinates { (-4.5,2.5)(-4.5,3.5)};

\draw [magenta, line width=1pt] plot [smooth, tension=1.3] coordinates { (-9+\dreta,5)(-7.5+\dreta,5.5)(-6+\dreta,5)};
\draw [magenta, line width=1pt] plot [smooth, tension=1.3] coordinates { (-9+\dreta,5)(-7.5+\dreta,4.5)(-6+\dreta,5)};
\draw [magenta, line width=1pt] plot [smooth, tension=1.3] coordinates { (-6+\dreta,5)(-4.5+\dreta,5.5)(-3+\dreta,5)};
\draw [magenta, line width=1pt] plot [smooth, tension=1.3] coordinates { (-6+\dreta,5)(-4.5+\dreta,4.5)(-3+\dreta,5)};
\draw [magenta, line width=1pt] plot [smooth, tension=1.3] coordinates { (-3+\dreta,5)(-1.5+\dreta,5.5)(0+\dreta,5)};
\draw [magenta, line width=1pt] plot [smooth, tension=1.3] coordinates { (-3+\dreta,5)(-1.5+\dreta,4.5)(0+\dreta,5)};

\draw [color= cyan, fill=cyan] (-9 + \dreta,5) circle (5pt);
\draw [color= cyan, fill=cyan] (-6 + \dreta,5) circle (5pt);
\draw [color= cyan, fill=cyan] (-3 + \dreta,5) circle (5pt);
\draw [color= cyan, fill=cyan] (0 + \dreta,5) circle (5pt);

\end{tikzpicture}
\caption{Example of a colorable associated graph.}
\label{fig:graph2}
\end{subfigure}
\end{center}
\caption{Examples of associated graphs.}
\label{fig:graphs}
\end{figure}

\begin{definition}
A \textbf{$2$-coloring of a graph} is a labeling (with only two labels) of the vertices of the graph such that no two adjacent vertices share the same label.
\end{definition}

\begin{definition}
Since not every graph admits a $2$-coloring, a  graph is called \textbf{$2$-colorable} if it admits a $2$-coloring.
\end{definition}

\subsection{$b^{2k}$-symplectic orientable surfaces}

We start by proving that only orientable surfaces admit $b^{2k}$-symplectic structures as expected, since for $m=2k$, $\omega$ is a positive section of $\Lambda^{2}(^b T^*S)$.
\begin{theorem}
If a compact surface admits a $b^{2k}$-symplectic structure then it is orientable.
\end{theorem}

\begin{proof}

The proof consists in building a collar of $b^{2k}$-Darboux neighborhoods with compatible orientations in a neighborhood of each connected component of $Z$.

Pick $(\tilde{S}, \tilde{Z})$ an orientable double cover of the $b^{2k}$-surface $(S, Z)$, with $\rho: \mathbb{Z}/2\mathbb{Z}\times \tilde{S} \rightarrow \tilde{S}$ the deck transformation.
For each point $p \in \tilde{Z}$ we can find a $b^{2k}$-Darboux neighborhood $U_p$ (by shrinking the neighborhood if necessary) which does not contain other points identified by $\rho$ ($\pi(U_p) \cong U_p$). Thus $V_p := \pi(U_p)$, and $\omega = \frac{1}{x^{2k}} dx \wedge dy$. Being $\omega$ of this type, it defines an orientation on $V_p\setminus \pi(\tilde{Z})$.
Since $\tilde{Z}$ is compact we can take a finite covering for such neighborhoods to define a collar $V$ of compatible orientations.
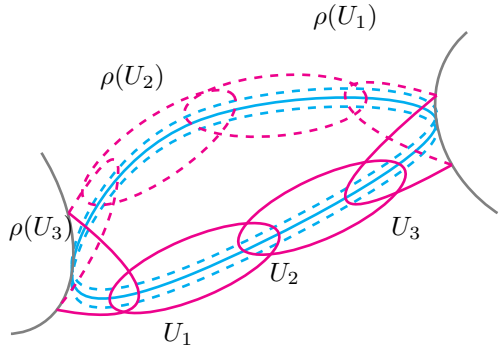
\begin{figure}[h!]

\begin{tikzpicture}[scale = 0.4]

\draw [cyan, line width=1pt] plot [smooth cycle, tension=0.8] coordinates { (-1,-2) (-5,-2) (-1,3) (7,3) };


\draw [cyan,dashed, line width=1pt] plot [smooth cycle, tension=0.8] coordinates { (-1,-2+0.3) (-5,-2+0.3) (-1,3+0.3) (7,3+0.3) }; 
\draw [cyan,dashed, line width=1pt] plot [smooth cycle, tension=0.8] coordinates { (-1,-2-0.3) (-5,-2-0.3) (-1,3-0.3) (7,3-0.3) }; 


\draw [gray, line width=1pt] plot [smooth, tension=1.2] coordinates { (-6,2)(-5,-2)(-7,-4) };
\draw [gray, line width=1pt] plot [smooth, tension=1.2] coordinates { (8,6)(7,3)(9,0) };


\draw [rotate = 23,magenta, line width=1pt] (-1.6,-1.3) ellipse (3cm and 1cm);
\draw [rotate = 25,magenta, line width=1pt] (3,-1.2) ellipse (3cm and 1cm);
\draw [magenta, line width=1pt] plot [smooth, tension=1.3] coordinates { (7,3.9)(4,.5)(7.5,1.6) };
\draw [magenta, dashed, line width=1pt] plot [smooth, tension=1.8] coordinates { (7,3.9)(4,3.8)(7.5,1.6) };
\draw [rotate = 5,dashed, magenta, line width=1pt] (2,3.4) ellipse (3cm and 1cm);
\draw [rotate = 35,dashed, magenta, line width=1pt] (-0.5,3) ellipse (3cm and 1cm);
\draw [magenta, line width=1pt] plot [smooth, tension=1.1] coordinates { (-5.1,0)(-2.8,-2.8)(-5.5,-3.2) };
\draw [magenta, dashed, line width=1pt] plot [smooth, tension=2.5] coordinates { (-5.1,0)(-3.5,1)(-5.5,-3.2) };


\node (b) at (-1.5,-4) {$U_1$};
\node (b) at (2,-2) {$U_2$};
\node (b) at (6,-.5) {$U_3$};
\node (b) at (4,6.5) {$\rho(U_1)$};
\node (b) at (-3,4.5) {$\rho(U_2)$};
\node (b) at (-6,-.5) {$\rho(U_3)$};

\end{tikzpicture}

\caption{A collar of compatible neighborhoods.}
\label{fig:bombs}
\end{figure}
Furthermore we can assume this covering to be symmetric (for each $U_p$  the image $\rho(U_p)$ is also one of the subsets of the covering).
This covering will be compatible with the orientation because $\rho$ preserves $\omega$.
The compatible orientations and the symmetric coverings descend to $(S,Z)$, thus defining an orientation in $(S,Z)$ inducing an orientation on $V\setminus Z$. By perturbing $\omega$ in $V$ we obtain a symplectic structure $\tilde{\omega}$ on $V$  and thus an orientation in $V$. Using the standard techniques of Radko \cite{radko} these can be glued to define an orientation via the symplectic form $\tilde{\omega}$ on the whole $S$. In the case $Z$ has more than one connected component we may proceed in the same way by isolating collar neighborhoods of each component. Thus proving that  $S$ is oriented.

\end{proof}

\subsection{$b^{2k+1}$-symplectic orientable surfaces}

\begin{theorem}\label{prescribed} Given a pair $(S,Z)$ with $S$ orientable, there exists a $b^{m}$-symplectic structure with critical set $Z$ whenever:

 \begin{enumerate}
   \item $m = 2k$,
   \item $m = 2k+1$  if only if the associated graph is $2$-colorable.
 \end{enumerate}

\end{theorem}

For the proof we will need the following,

\begin{lemma}[\textbf{Weinstein normal form theorem}]\label{WeinsteinNormalFormTheorem} Let $L$ be a Lagrangian submanifold of a symplectic manifold $(M, \omega)$, then there exists a neighborhood of $L$, $U_L$ which is symplectomorphic to a neighborhood of the zero section of the cotangent bundle of $T^{*}L$ endowed with the symplectic form $-d\lambda$ with $\lambda$ the Liouville one form.
\end{lemma}

\begin{proof} (of Theorem \ref{prescribed})

Let $C_1, \ldots, C_r$ be the connected components of $S\setminus Z$, let $Z_1, \ldots, Z_s$ the connected components of $Z$ and let $\mathcal{U}(Z_1), \ldots, \mathcal{U}(Z_s)$ tubular neighborhoods of the connected components.

We can prove this using a 3-step proof.

\begin{enumerate}

\item \textbf{Using Weinstein normal form theorem.} By virtue of Lemma \ref{WeinsteinNormalFormTheorem}, each tubular neighborhood $\mathcal{U}(Z_i)$ can be identified with a zero section of the cotangent bundle of $Z_i$. Now replace, the cotangent bundle of $Z_i$ by the $b^{m}$-cotangent bundle of $Z_i$. The neighborhood of the zero section of the $b^m$-cotangent bundle has a $b^m$-symplectic structure $\omega_{\mathcal{U}(Z_i)}$.

\item \textbf{Constructing compatible orientation using the graph.} We assign a couple of signs to each tubular neighborhood using the sign of the Liouville volume of $\omega_{\mathcal{U}(Z_i)}$. Note that the sign does not change for $k$ even, but it changes for $k$ odd. Observe that we can apply Moser's trick to glue two rings that share some $C_j$ if and only if the sign of the two rings match on this component. In terms of coloring graphs, this condition  translates into one vertex only having one color. In other words the condition of adjacent signs matching determines the color of a vertex.

    Now, let us consider separately the odd an even cases:

\begin{enumerate}
    \item For $b^{2k}$ the color of adjacent vertices must coincide. And hence we have no additional constraint on the topology of the graph.
    \item In the $b^{2k+1}$ case the sign of two adjacent vertices must be different. Then, we have to impose the associated graph to be $2$-colorable.
\end{enumerate}
%
%
%
%

\item \textbf{Gluing.} Now we may glue back this neighborhood to $S\setminus \mathcal{U}(Z)$ in such a way that the symplectic structures fit on the boundary, by means of the Moser's path method.

%
%
%
%
%
%
%
%
\end{enumerate}
\end{proof}

Given a $b^{2k+1}$-symplectic structure $\omega$ on a $b$-surface $(S,Z)$ one can obtain a $2$-coloring of the associated graph by assigning to each connected component $C_i$ of $S\setminus Z$ the `color' $\text{sign}(\int_{C_i} \omega)$.

Another way to construct $b^{2k}$-structures on a surface is to use decomposition theorem as connected sum of $b^{2k}$-spheres (\ref{ex:sphere}) and $b^{2k}$-torus (\ref{ex:torus}). The drawback of this construction is that it is harder to adapt  having fixed a prescribed $Z$.

\subsection{$b^{2k+1}$-symplectic non-orientable surfaces}

\begin{remark}
Since a $b$-map on a $b$-manifold $(M,Z)$ sends $Z$ to $Z$, it has to send connected components of $M\setminus Z$ to connected components of $M\setminus Z$.
\end{remark}

\begin{definition}
We say that a $b$-map $\rho$ \textbf{inverts colors} of the associated graph if $\text{sign}(\int_{C_i} \omega) = -\text{sign}(\int_{\rho(C_i)} \omega)$.
\end{definition}

Observe also that a $b$-map that inverts colors is also orientation reversing. Because the orientations defined in $S$ by $\omega$ are different in connected components with different colors. One may see this by looking at a neighborhood of $Z$ in coordinates.

\begin{theorem}
Any pair $(S,Z)$ with  $S$ a non-orientable surface and  $Z$ a hypersurface admits a $b^{2k+1}$-symplectic if and only if the following two conditions hold:

\begin{enumerate}
	\item the graph of the covering $(\tilde{S},\tilde{Z})$ is 2-colorable and
	\item the deck transformation inverts colors of the graph obtained in the covering.
\end{enumerate}
\end{theorem}

\begin{proof}
	
Apply Theorem \ref{prescribed} to endow the covering $(\tilde{S},\tilde{Z})$ with a $b^{2k+1}$-symplectic structure, if the form obtained is invariant by the deck transformations, then we may apply the equivariant $b^m$-Moser theorem \ref{emt} considering the deck transformations as group action, thus obtaining a $b^{2k+1}$-symplectic structure on $(\tilde{S},\tilde{Z})$ and then we are done.
	
Now, let us assume that the $b^{2k+1}$-form $\omega$ obtained via theorem \ref{prescribed} is not invariant by deck transformations. We will note the deck transformation induced by $-\text{Id}$ as $\rho$.
Observe that
\begin{equation}
\text{sign}\left(\int_{C_i} \rho^* \omega\right) = - \text{sign}\left(\int_{\rho(C_i)} \omega\right) = + \text{sign}\left(\int_{C_i} \omega\right).
\end{equation}
The first equality is due to $\rho$ changing orientations and the second one is due to $\rho$ inverting colors. Then the pullback of $\omega$ has the same sign as $\omega$, and hence $\omega + \rho^*(\omega)$ is a non-degenerate $b^{2k+1}$-form that is invariant under the action of $\rho$, and we can take it down to the quotient. Hence a $b^{2k+1}$-symplectic structure is obtained on $(S,Z)$.

\end{proof}

\begin{example}
Let us illustrate what is happening in the previous proof with an example. Take the sphere having the equator as critical set and endowed with the $b$-symplectic form $\omega = \frac{1}{h} dh\wedge d\theta$. Let us call the north hemisphere $C_1$ and the south hemisphere $C_2$, and let $\rho$ be the antipodal map. Look at the coloring of the graph (a path graph of length two):
\begin{equation}
\text{sign}(C_1) = \text{sign}\left(\int_{C_1}\omega\right) = \text{sign}\left(\lim_{\varepsilon \rightarrow 0} \int_{-\pi}^{\pi} \int_\varepsilon^{1} \frac{1}{h}dh\wedge d\theta\right) = \text{sign}(\lim_{\varepsilon \rightarrow 0}-2\pi \log|\varepsilon|)
\end{equation}
which is positive. And
\begin{equation}
\text{sign}(C_2) = \text{sign}\left(\int_{C_2}\omega\right) = \text{sign}\left(\lim_{\varepsilon \rightarrow 0} \int_{-\pi}^{\pi} \int_{-1}^{-\varepsilon} \frac{1}{h}dh\wedge d\theta\right) = \text{sign}(\lim_{\varepsilon \rightarrow 0}2\pi \log|\varepsilon|)
\end{equation}
which is negative. Then,
\begin{equation}
\int_{C_1} \rho^* \omega = \lim_{\varepsilon \rightarrow 0} \int_{-\pi}^{\pi} \int_{\epsilon}^{1} \rho^*\left(\frac{1}{h}dh\wedge d\theta\right) = \lim_{\varepsilon \rightarrow 0} \int_{-\pi}^{\pi} \int_{\epsilon}^{1} \frac{1}{-h}d(-h)\wedge d\theta = \int_{C_1} \omega.
\end{equation}

In this case $\omega$ was already invariant, but one can observe that if $\rho$ inverts colors then the signs of the form and the pullback are the same.
\end{example}

\begin{example}
One may ask why the condition of inverting colors is necessary. Next we provide an example where $b^{2k+1}$-structures can be exhibited on the double cover but cannot be projected to induce a $b^{2k+1}$-structure on the non-orientable surface.

Consider the Example \ref{example:projective} in Section \ref{toyexamples}. If one translates the critical set in the $h$ direction in the projective space, the double cover is still the sphere, but instead of $Z$ being the equator, $Z$ consists of different meridians $\{h = h_0\}$ and $\{h = -h_0\}$.

Observe that the associated graph of this double cover is a path graph of length $3$, that can be easily $2$-colored. Take a generic $b^{2k+1}$-form $\omega = f(h, \theta) dh\wedge d\theta$, and look at the poles $N, S$. $\text{sign}(f(N)) = \text{sign}(f(S))$ because of the $2$-coloring of the graph. But $\rho^*(\omega)|_N = f(\rho(N))d(-h)\wedge d\theta = -f(S)dh\wedge d\theta$. Then $\text{sign}(\rho^*(\omega)) \neq \text{sign}(\rho^*(\omega))$, and hence $\omega$ can not be invariant for $\rho$.
\end{example}

\section{Desingularization of surfaces}\label{sec:desingularization}

Let $(M,\omega)$ be a $b^{2k}$-symplectic manifold and let $Z$ be the critical set and $U_\epsilon$ an $\epsilon$-neighborhood of $Z$.

The aim of this section is to use the desingularization formulas \cite{evajonathanvictor} in the case of surfaces. We end seeing that if $[\omega_1] = [\omega_2]$ in the $b^{2k}$-cohomology, then the desingularization of the two forms also is in the same class $[\omega_{1\epsilon}] = [\omega_{2\epsilon}]$. But the converse is not true: it is possible to find  different classes of $b^{2k}$-forms that go the same class when desingularized.

Let us  briefly recall how the desingularization is defined and  the main result in \cite{evajonathanvictor}.
Recall that we can express the $b^{2k}$-form as:
$$\omega = \frac{dx}{x^{2k}}\wedge \left(\sum_{i=0}^{2k-1}x^i\alpha_i\right) + \beta.$$

\begin{definition}The \textbf{$f_\epsilon$-desingularization} $\omega_\epsilon$ form of $\omega$ is:
	$$\omega_\epsilon = df_\epsilon \wedge \left(\sum_{i=0}^{2k-1}x^i\alpha_i\right) + \beta.
	$$
	Where $f_\epsilon(x)$ is defined as $\epsilon^{-(2k -1)}f(x/\epsilon)$.
	And $f \in \mathcal{C}^\infty(\mathbb{R})$ is an odd smooth function satisfying $f'(x) > 0$ for all $x \in \left[-1,1\right] $ and satisfying outside that
\begin{equation}
f(x) = \begin{cases}
\frac{-1}{(2k-1)x^{2k-1}}-2& \text{for} \quad x < -1,\\
\frac{-1}{(2k-1)x^{2k-1}}+2& \text{for} \quad x > 1.\\
\end{cases}
\end{equation}
\begin{figure}[h!]
	\centering
	\includegraphics[scale=.7]{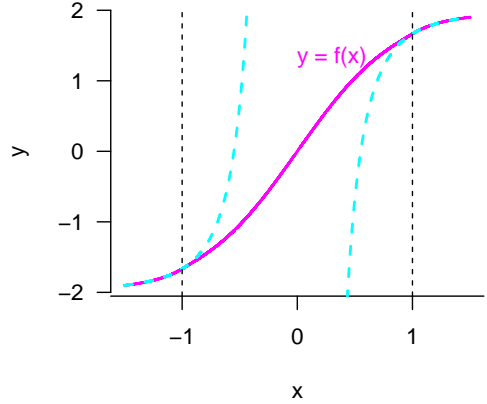}
	\caption{Smooth and odd extension of $f$ inside the interval $[0,1]$ such that $f' \neq 0$.}
	\label{fig:deblogging}
\end{figure}
\end{definition}

\begin{theorem}[\textbf{Desingularization}, \cite{evajonathanvictor}]
	The $f_\epsilon$-desingularization $\omega_{1\epsilon}\epsilon$ is symplectic. The family $\omega_\epsilon$ coincides with the $b^{2k}$-form $\omega$ outside an $\epsilon$-neighbourhood of $Z$. The family of bivector fields $\omega_\epsilon^{-1}$ converges to the structure $\omega^{-1}$ in the $C^{2k-1}$-topology as $\epsilon \rightarrow 0$.
\end{theorem}

Next we apply our classification scheme and see how it behaves using the desingularization procedure\footnote{also called deblogging}:

\begin{theorem}[\textbf{Compatibility of Classification and deblogging}, \cite{evajonathanvictor}] If $[\omega_1] = [\omega_2]$ in $b^{2k}$-cohomology then $[\omega_{1\epsilon}] = [\omega_{2\epsilon}]$ in De Rham cohomology. The other implication does not hold: there are different classes of $b^{2k}$-forms that desingularize to the same class.
\end{theorem}

\begin{remark} This theorem asserts that, for surfaces, equivalent $b^{2k}$-symplectic structures get mapped to equivalent symplectic structures under the desingularization procedure. Though non-equivalent $b^{2k}$-symplectic structures might get mapped to equivalent symplectic structures via deblogging.

\end{remark}
\begin{proof}

In order to compute the class of the desingularization we calculate the integral of the desingularized form over the whole manifold.
We are going to proceed in two steps. Firstly we are going to compute the integral of the desingularization inside the $\epsilon$-neighborhood $U_\epsilon$ of $Z$, and then we compute it outside.

Using the expression of $\omega_\epsilon$ we compute:

\begin{tabular}{rcl}
$\displaystyle \int_{U_\epsilon} \omega_\epsilon$ & $= $& $\displaystyle \int_{U_\epsilon}  df_\epsilon \wedge \left(\sum_{i=0}^{2k-1}x^i\alpha_i\right) + \int_{U_\epsilon}\beta$\\
& $ = $ & $\displaystyle \epsilon^{-(2k-1)}\int_{U_\epsilon}  df(x/\epsilon) \wedge \left(\sum_{i=0}^{2k-1}x^i\alpha_i\right) + \int_{U_\epsilon}\beta$\\
& $ = $ & $\displaystyle \epsilon^{-2k}\int_{U_\epsilon}  \frac{df(x/\epsilon)}{dx} dx \wedge \left(\sum_{i=0}^{2k-1}x^i\alpha_i\right) + \int_{U_\epsilon}\beta$\\
& $ = $ & $\displaystyle \epsilon^{-2k} \sum_{i=0}^{2k-1} \int_{-\epsilon}^{+\epsilon} \frac{df(x/\epsilon)}{dx} x^i dx \int_{Z} \alpha_i + \int_{U_\epsilon}\beta.$\\
\end{tabular}

Then, because $f$ is an odd function, $df(x/\epsilon)/dx$ is even and hence the integral $\int_{-\epsilon}^{+\epsilon} \frac{df(x/\epsilon)}{dx} x^i dx$ is going to be different from $0$ if $i$ is even. Thus,

\begin{tabular}{rcl}
$\displaystyle \int_{U_\epsilon} \omega_\epsilon$ & $ = $ & $\displaystyle \epsilon^{-2k} \sum_{i=1}^{k-1} \int_{-\epsilon}^{+\epsilon} \frac{df(x/\epsilon)}{dx} x^{2i} dx \int_{Z} \alpha_{2i} + \int_{U_\epsilon}\beta$.\\
\end{tabular}

Recall that outside the $\epsilon$-neighborhood the desingularization $\omega_{\epsilon}$ coincides with the $b^{2k}$-symplectic form $\omega$.
Moreover, let us define $U$ a tubular neighborhood of $Z$ containing $U_\epsilon$, (assume $U = [-1,1]\times Z$). Following the computations in \cite{scott} we obtain,

\begin{tabular}{rcl}
$\displaystyle \int_{M\setminus U_\epsilon} \omega_\epsilon$ & $= $& $\displaystyle \int_{M\setminus U_\epsilon} \omega $\\
& $ = $ & $\displaystyle \int_{M\setminus U} \omega + \int_{U\setminus U_\epsilon} \omega$\\
& $ = $ & $\displaystyle \int_{M\setminus U} \omega + \left(\int_{U \setminus U_\epsilon} \beta + \sum_{i=1}^k \frac{-2}{2i-1}\int_Z \alpha_{2i}  \right) + \sum_{i=1}^{k} \left(\frac{2}{2i-1}\int_Z \alpha_{2i}\right)\epsilon^{2i-1}.$\\
\end{tabular}

Also in \cite{scott}, the Liouville Volume that appears at Theorem \ref{thm:scott} is defined as follows:

$$
P_{[\omega]}(0) = \left( \int_{M\setminus U} \omega + \int_{U} \beta + \sum_{i = 1}^{k} \left(\frac{-2}{2i-1}\right)\int_Z \alpha_{2i}\right).
$$

Observe that in \cite{scott} the term $\int_{M\setminus U} \omega$ does not appear but if we add it,  $P_{[\omega]}(t)$ still has the desired properties.
We can now proceed with the computations:

\begin{tabular}{rcl}
$\displaystyle \int_{M} \omega_\epsilon$ & $ = $& $\displaystyle \epsilon^{-2k} \sum_{i=1}^{k-1} \int_{-\epsilon}^{+\epsilon} \frac{df(x/\epsilon)}{dx} x^{2i} dx \int_{Z} \alpha_{2i} + \int_{U_\epsilon}\beta$\\
& $+$& $ \displaystyle \int_{M\setminus U} \omega + \left(\int_{U \setminus U_\epsilon} \beta + \sum_{i=1}^k \frac{-2}{2i-1}\int_Z \alpha_{2i}  \right) + \sum_{i=1}^{k} \left(\frac{2}{2i-1}\int_Z \alpha_{2i}\right)\epsilon^{2i-1}$\\
& $=$& $\displaystyle \epsilon^{-2k} \sum_{i=1}^{k-1} \int_{-\epsilon}^{+\epsilon} \frac{df(x/\epsilon)}{dx} x^{2i} dx \int_{Z} \alpha_{2i}$\\
& $+$& $\displaystyle \underbrace{\int_{M\setminus U} \omega + \left(\int_{U} \beta + \sum_{i=1}^k \frac{-2}{2i-1}\int_Z \alpha_{2i}  \right)}_{P_{[\omega]}(0)} + \sum_{i=1}^{k} \left(\frac{2}{2i-1}\int_Z \alpha_{2i}\right)\epsilon^{2i-1}.$\\
\end{tabular}

This integral only depends on the classes $[\alpha_i]$ and the Liouville Volume $P_{[\omega]}(0)$, which are determined by (and determine) the class of $[\omega]$. So, two $b^{2k}$-forms on the same cohomology class, determine the same cohomology class when desingularized. 

The converse is not true, because in the previous formula we only have terms $[\alpha_i]$ with $i$ even. As a consequence, if two forms differ only in the odd terms have the same desingularized forms assuming the auxiliary function $f$ in the desingularization process is the same.

\end{proof}

\begin{example}

Consider $S^2$ with the following two $b^2$-symplectic structures:
\begin{equation}\omega_1 = \frac{1}{h^2} dh\wedge d\theta, \qquad \omega_2 = \left(\frac{1}{h} + \frac{1}{h^2}\right)dh\wedge d\theta = \frac{1}{h^2}dh\wedge(h d\theta + d\theta).\end{equation}

Observe that for $\omega_1$, the forms on the Laurent series are $\alpha_0^1 = d\theta$ and $\alpha_1^1 = 0$, while for $\omega_2$ they are $\alpha_0^2 = d\theta$ and $\alpha_1^2 = d\theta$. Then $\int_Z \alpha_1^1 = 0 \neq \int_Z \alpha_1^2 = 2\pi$, and hence $[\alpha_1^1] \neq [\alpha_1^2]$ and $[\omega_1] \neq [\omega_2]$. The desingularized expressions of those forms are given by:
\begin{equation}
\omega_{1\epsilon} = \begin{cases}
               \displaystyle\frac{df_\epsilon(h)}{dh} dh\wedge d\theta & \text{if } |h| \leq \epsilon, \\
               \omega_1 & \text{otherwise},
             \end{cases}
 \qquad \text{and} \qquad
\omega_{2\epsilon} = \begin{cases}
               \displaystyle \frac{df_\epsilon(h)}{dh} dh\wedge (h d\theta + d\theta) & \text{if } |h| \leq \epsilon, \\
               \omega_2 & \text{otherwise}.
             \end{cases}
\end{equation}

Let us compute the classes of $\omega_{1\epsilon}$ and $\omega_{2\epsilon}$.

\begin{tabular}{rcl}
  $\displaystyle \int_{S^2} \omega_{2\epsilon}$ & $=$ & $\displaystyle \int_{S^2 \setminus U_\epsilon} \omega_2 + \int_{U_\epsilon} \frac{df_\epsilon (h)}{dh}(hd\theta + d\theta)$\\
   & $=$ & $\displaystyle \int_{S^2 \setminus U_\epsilon} \frac{1}{h^2} dh\wedge (hd\theta + d\theta) + \int_{U_\epsilon} \frac{df_\epsilon (h)}{dh}(d\theta) + \underbrace{\int_{U_\epsilon} \frac{df_\epsilon (h)}{dh}(h d\theta)}_{=0}$\\
   & $=$ & $\displaystyle \int_{S^2 \setminus U_\epsilon} \frac{1}{h^2} dh\wedge d\theta + \underbrace{\int_{S^2 \setminus U_\epsilon} \frac{1}{h} dh\wedge d\theta}_{=0} + \int_{U_\epsilon} \frac{df_\epsilon (h)}{dh}(d\theta)$\\
   & $=$ & $\displaystyle \int_{S^2} \omega_{1\epsilon}$.\\

\end{tabular}

Let us consider the action of $S^1$ over $S^2$ given by
  $\phi :S^1\times S^2 \rightarrow S^2:(t, (h,\theta)) \mapsto (h, \theta + t)$.
Observe that both $\omega_1$ and $\omega_2$ are invariant under the previous action. Moreover, their desingularizations are also invariant.
\end{example}

\section{Constructions and classification of $b^m$-Nambu structures}\label{sec:nambu}

Let us start defining a $b^m$-Nambu structure of top degree,
\begin{definition}\label{defnamb1}
A $b^m$-Nambu structure of top degree on a pair $(M^n, Z)$ with $Z$  a smooth hypersurface  is given by a smooth $n$-multivector field  $\Lambda$  such that exists a local system of coordinates for which
\begin{equation}
\Lambda= x_1^m \frac{\partial}{\partial x_1}\wedge\ldots\wedge\frac{\partial}{\partial x_n}
\end{equation}
and $Z$ is locally defined by $x_1 = 0$.
\end{definition}
Dualizing the local expression of the Nambu structure we obtain the form

\begin{equation}
\Theta=\frac{1}{x_1^m} dx_1 \wedge \ldots \wedge dx_n
\end{equation}

\noindent (which is not a smooth de Rham form), but it is a $b^m$-form of degree $n$ defined on a $b^m$-manifold. As it is done in \cite{guimipi2}, we can check that this dual form is non-degenerate. So we may define a $b^m$-Nambu form as follows.

Mimicking the same condition as for $b^m$-symplectic forms we can talk about non-degenerate $b^m$-forms of top degree. This means that seen as a section of $\Lambda^n(^b T^*M)$ the form does not vanish.

\textbf{Notation:} We will denote by $\Lambda$ the Nambu multivectorfield and by $\Theta$ its dual.

\begin{definition}\label{defnamb2}
A $b^m$-Nambu form is a non-degenerate $b^m$-form of top degree.
\end{definition}

We first include a collection of motivating examples, and then prove an equivariant classification theorem.

\subsection{Examples} 

\begin{enumerate}

\item\textbf{$b^m$-symplectic surfaces:}\label{ex:symplecticnambu} Any $b^m$-symplectic surface is a $b^m$-Nambu manifold with Nambu structure of top degree.

\item \textbf{$b^m$-symplectic manifolds as $b^{m}$-Nambu manifolds:} Let  $(M^{2n}, \omega)$ be a $b^m$-symplectic manifold, then
$(M^{2n}, \underbrace{\omega\wedge \ldots\wedge \omega}_{n})$ is automatically $b^{m}$-Nambu.

\item \textbf{Orientable manifolds:} Let $(M^{n}, \Omega)$ be any orientable manifold (with $\Omega$ a volume form) and let $f$ be a defining function for $Z$, then $(1/f^m)\omega$ defines a $b^m$-Nambu structure of top degree having $Z$ as critical set.

 Any Nambu structure can be written in this way if the hypersurface can be globally described as the vanishing set of a smooth function.

%
%
%
%

 \item\textbf{Spheres:} In \cite{david}, it was given special importance to the example $(S^n, \sqcup_i S_i^{(n-1)})$ because of the Schoenflies theorem\footnote{The nature of this theorem is purely topological in dimension equal or greater than four, and so is its construction.}, which imposes the associated graph to be a tree. The nice feature of this example is that $O(n)$ acts on the $b^m$-manifold  $(S^n, S^{(n-1)})$, and it makes sense to consider its classification under these symmetries. This also works for other homogeneous spaces of type $(G_1/G_2, G_2/G_3)$ with $G_2$ and $G_3$ with codimension 1 in $G_1$ and $G_2$ respectively.

\end{enumerate}
\subsection{ $b^m$-Nambu structures of top degree and orientability}

As we did in the case of $b^{2k}$-symplectic structures we can prove the following theorem:

\begin{theorem} A compact $n$-dimensional manifold $M$ admitting a $b^{2k}$-Nambu structure  is orientable.
\end{theorem}

\begin{proof}
  The proof consists in building a collar of $b^m$-Darboux charts for the $b^{2k}$-Nambu structure (such that in local coordinates the Nambu structure can be written as $x_1^{2k} \frac{\partial}{\partial x_1}\wedge\ldots\wedge\frac{\partial}{\partial x_n}$) with compatible orientations in a neighborhood of each connected component of $Z$.

Consider a 2:1 orientable covering $(\tilde{M}, \tilde{Z})$ of the manifold and denote by  $\rho: \mathbb{Z}/2\mathbb{Z}\times \tilde{M} \rightarrow \tilde{M}$ the deck transformation.
For each point $p \in \tilde{Z}$ we provide a Darboux  neighborhood $U_p$ which does not contain other points identified by $\rho$. Thus $U_p \cong \pi(U_p) =: V_p$, and $\Theta = \frac{1}{x^{2k}} dx_1 \wedge \ldots \wedge dx_n$. This form defines an orientation on $V_p\setminus \pi(Z)$.
Take a symmetric covering of such neighborhoods to define a collar of $Z$ with compatible orientations, and compatible with the covering.
The compatible orientations and the symmetric coverings descend to $(M,Z)$, thus defining an orientation in $(M,Z)$.
Thus, we have an orientation in $V\setminus Z$. By perturbing $\omega$ in $V$ we obtain a volume form on $V$, $\tilde{\omega}$, and thus an orientation in $V$. These can be glued to define an orientation via the volume form $\tilde{\Theta}$ on the whole $M$ proving that $M$ is oriented.
\end{proof}

\subsection{Classification of $b^m$-Nambu structures of top degree and $b$-cohomology}

We present the definitions contained in   \cite{david}  of  modular period attached to the connected component of an orientable Nambu structure using the language of $b^m$-forms.

Let $\Theta$ be the dual to the multivectorfield $\Lambda$ defining a Nambu structure.
From the general decomposition of $b^m$-forms as it was set in Equation \ref{eq:decomposition} we may write.

\begin{equation}\label{eq:decompositionnambu}
\Theta=\Theta_0\wedge\frac{df}{f^m}, \text{ with } \Theta_0\in\Omega^{n-1}(M).
\end{equation}

This decomposition is valid in a neighborhood of $Z$ whenever the defining function is well-defined. Otherwise for non-orientable manifolds a similar decomposition can be proved by replacing the defining function $f$ by an adapted distance (see \cite{marcutosorno2}).

With this language in mind, the
the \textbf{modular $(n-1)$-vector field} in \cite{david} of $\Theta$ along $Z$ is the dual of the form $\Theta_0$ in the decomposition above which is indeed  the
\textbf{modular $(n-1)$-form} along $Z$ in \cite{david}.

Recall from \cite{david} in our language:\begin{definition} The \textbf{modular period} $T_\Lambda^Z$ of the component $Z$ of the
zero locus of $\Lambda$ is
\[ T_\Lambda^Z :=
\int_Z \Theta_0>0.\]
\end{definition}

In fact, this positive number determines the Nambu structure in a
neighborhood of $Z$ up to isotopy as it was proved in \cite{david}.

The following theorem gives a classification of $b^m$-Nambu structures.

We will prove a more general result for manifolds, which are not necessarily orientable, admitting a Nambu structure of top degree. We do it using the dual Nambu forms associated to the multivectorfields.

\begin{theorem}\label{thm:bnnambu} Let $\Theta_0$ and $\Theta_1$ be two $b^m$-Nambu forms of degree $n$ on a compact orientable manifold $M^n$. If $[\Theta_0] = [\Theta_1]$ in $b^m$-cohomology then there exists a diffeomorphism $\phi$ such that $\phi^{*}\Theta_1 = \Theta_0$.
\end{theorem}

\begin{proof}
  We will apply the techniques of \cite{moser} with the only difference that we work with $b^m$-volume forms instead of volume forms.

  Since $\Theta_0$  and $\Theta_1$ are non-degenerate $b^m$-forms both of them are a multiple of a volume form and thus the linear path
  $\Theta_t= (1-t)\Theta_0+ t\Theta_1$ is a path of non-degenerate $b^m$-forms.

  Since $\Theta_0$  and $\Theta_1$  determine the same cohomology class:
  $$\Theta_1-\Theta_0=d\beta$$ with $d$ the $b^m$-De Rham differential and $\beta$ a $b^m$-form of degree $n-1$.

  Now consider the Moser equation:
  \begin{equation}\label{moremosertrick}
\iota_{X_t}\Theta_t=-\beta.
\end{equation}

Observe that since $\beta$ is a $b^m$-form and $\Theta_t$ is non-degenerate. The vector field $X_t$ is a $b^m$-vector field.
Let  $\phi_t$  be the t-dependent flow integrating $X_t$.

 The $\phi_t$ gives the desired diffeomorphism $\phi_t:M\to M$, leaving $Z$ invariant (since $X_t$ is tangent to $Z$) and $\phi_t^*\Theta_t=\Theta_0$.
\end{proof}

In particular we recover the classification of $b$-Nambu structures of top degree in \cite{david}:

\begin{theorem}[\textbf{Classification of $b$-Nambu structures of top degree, \cite{david}}]
\label{thm:david} A generic $b$-Nambu structure $\Theta$ is
determined, up to orientation preserving diffeomorphism, by the following three invariants:
the diffeomorphism type of the oriented pair $(M,Z)$,
the modular periods and  the regularized Liouville volume.
\end{theorem}
By Theorem \ref{thm:Mazzeo-Melrose},
$$  {^{b}}\hspace{-1pt}H^n(M) \cong H^n(M) \oplus  H^{n-1}(Z).$$

The first term on the right hand side is the Liouville volume image by the De Rham theorem, as it was done in \cite{gmps} for $b$-symplectic forms.  The second term collects the periods of the modular vector field.  So if the three invariants coincide then they determine the same $b$-cohomology class.

In other words, the statement in \cite{david} is equivalent to  the following theorem in the language of $b$-cohomology.

\begin{theorem}
\label{thm:davidcohom} Let $\Theta_1$ and $\Theta_2$ be two $b$-Nambu forms on an orientable manifold $M$. If $[\Theta_1] = [\Theta_2]$ in $b$-cohomology then there exists a diffeomorphism $\phi$ such that $\phi^{*}\Theta_1 = \Theta_2$.
\end{theorem}

This global Moser theorem for $b^m$-Nambu structures admits an equivariant version,

\begin{theorem} Let $\Theta_0$ and $\Theta_1$ be two $b^m$-Nambu forms of degree $n$ on a compact orientable manifold $M^n$ and let $\rho:G\times M\longrightarrow M$ be a compact Lie group action  preserving both $b^m$-forms. If $[\Theta_0] = [\Theta_1]$ in $b^m$-cohomology then there exists an equivariant diffeomorphism $\phi$ such that $\phi^{*}\Theta_1 = \Theta_0$.

\end{theorem}

\begin{proof}
As in the former proof, write
 $$\Theta_1-\Theta_0=d\beta$$ with $d$ the $b^m$-De Rham differential and $\beta$ a $b^m$-form of degree $n-1$.
Observe that the path $\Theta_t= (1-t)\Theta_0+ t\Theta_1$ is a path of invariant $b^m$-forms.

  Now consider Moser's equation:
  \begin{equation}\label{moremosertrick}
\iota_{X_t}\Theta_t=-\beta.
\end{equation}

Since $\Theta_t$ is invariant we can find an invariant $\tilde{\beta}$. For instance take $\tilde{\beta}=\int_G \rho_g^*(\beta)d\mu$ with $\mu$ a de Haar measure on $G$ and $\rho_g$ the induced diffeomorphism $\rho_g(x):=\rho(g,x)$.

Now replace $\beta$ by $\tilde{\beta}$ to obtain,

  \begin{equation}\label{moremosertrickinv}
\iota_{X_t^G}\Theta_t=-\tilde{\beta}
\end{equation}

\noindent  with $X_t^G=\int_G {\rho_g}_* X_t d\mu$.
 The vector field $X_t^G$ is an invariant $b$-vector field.
Its flow  $\phi_t^G$  preserves the action  and $\phi_t^{G}{^*}\Theta_t=\Theta_0$.

\end{proof}

Playing the equivariant $b^m$-Moser trick as we did in Section \ref{sec:nonorientable} we obtain,
\begin{corollary}\label{thm:bnnambunonorientable} Let $\Theta_0$ and $\Theta_1$ be two $b^m$-Nambu forms of degree $n$ on a manifold $M^n$ (that has not to be orientable). If $[\Theta_0] = [\Theta_1]$ in $b^m$-cohomology then there exists a diffeomorphism $\phi$ such that $\phi^{*}\Theta_1 = \Theta_0$.
\end{corollary}


\begin{thebibliography}{ABCD}


\bibitem[AR1]{arnold}
 Arnold, V. I.\emph{ Remarks on Poisson structures on a plane and on other powers of volume elements.} (Russian) Trudy Sem. Petrovsk. No. 12 (1987), 37--46, 242; translation in J. Soviet Math. 47 (1989), no. 3, 2509--2516.

\bibitem[AR2]{arnold2}Arnold, V. I. \emph{Critical point of smooth functions.} Vancouver Intern. Congr. of Math.,
     1974, vol.1, 19--39.

\bibitem[Ca]{cavalcanti} G. \ Cavalcanti,\emph{Examples and counter-examples of log-symplectic manifolds.}, preprint 2013, arxiv 1303.6420.




\bibitem[DKM]{dkm} A.\ Delshams, A. \ Kiesenhofer and E. \ Miranda, \emph{Examples of integrable and non-integrable systems on singular symplectic manifolds}, Journal of Geometry and Physics (2016), in press,       doi:10.1016/j.geomphys.2016.06.011.



\bibitem[FMM]{frejlichmartinezmiranda} P.\ Frejlich, D. \ Mart\'{\i}nez and E. Miranda, \emph{A note on the symplectic topology of $b$-symplectic manifolds, } arXiv:1312.7329, to appear at Journal of Symplectic Geometry.


 \bibitem[GL]{gualtieri} M.\ Gualtieri and S.\ Li \emph{Symplectic groupoids of log symplectic
manifolds},   Int. Math. Res. Notices,  First published online March 1, 2013 doi:10.1093/imrn/rnt024.

\bibitem[GLPR]{gualtierietal} M.\ Gualtieri, S.\ Li, A.\ Pelayo, T.\ Ratiu, \emph{The tropical momentum map: a classification of toric log symplectic manifolds}, Math. Ann. (2016). doi:10.1007/s00208-016-1427-9.



\bibitem[GMP]{guimipi} V.\ Guillemin, E.\ Miranda and A.R.\ Pires, \emph{Codimension one symplectic foliations and regular Poisson structures}, Bulletin of the Brazilian Mathematical Society, New Series 42 (4), 2011, pp.\ 1--17.


\bibitem[GMP2]{guimipi2} V.\ Guillemin, E.\ Miranda, A.R.\ Pires, \emph{Symplectic and Poisson geometry on $b$-manifolds}, \emph{Adv. Math.} \textbf{264}, pp.\ 864-896 (2014).

    \bibitem[GMPS]{gmps} V. \ Guillemin, E. \ Miranda, A. R. \ Pires and G. \ Scott, \emph{Toric actions on $b$-symplectic manifolds},  Int Math Res Notices (2015) 2015 (14): 5818-5848. doi: 10.1093/imrn/rnu108 .

\bibitem[GMW]{evajonathanvictor} V. \ Guillemin, E. \ Miranda and J. \ Weitsman, \emph{Desingularizing $b^m$-symplectic structures}, {\tt arXiv:1512.05303}.

\bibitem[KM]{km} A. \ Kiesenhofer and E. \ Miranda, \emph{Cotangent models for integrable systems},   Comm. Math. Phys. 350 (2017), no. 3, 1123--1145. .

\bibitem[KMS]{kms} A. \ Kiesenhofer, E. \ Miranda and G. \ Scott, \emph{Action-angle variables and a KAM theorem for b-Poisson manifolds},  J. Math. Pures Appl. (9) 105 (2016), no. 1, 66--85.



\bibitem[MO2]{marcutosorno2} I. \ Marcut and B. \ Osorno, \emph{ Deformations of log-symplectic structures
},  J. Lond. Math. Soc. (2) 90 (2014), no. 1, 197--212.

\bibitem[MD]{mcduff} D. \ McDuff, \emph{Examples of symplectic structures}. Invent. Math. 89 (1987), 13–36.

\bibitem[M]{moser} J. \ Moser, \emph{ On the volume elements on a manifold.}
Trans. Amer. Math. Soc. 120 1965 286--294.
\bibitem[Mt]{david} D.\ Martinez-Torres, \emph{Global classification of generic multi-vector fields of top degree}, Journal of the London Mathematical Society, 69:3, 2004, pp.\ 751--766.

\bibitem[Me]{melrose} R.\ Melrose, \emph{Atiyah-Patodi-Singer Index Theorem} (book), Research Notices in Mathematics, A.K.\ Peters, Wellesley, 1993.


\bibitem[MS]{mirandascott} E. \ Miranda and G. \ Scott,  \emph{The geometry of E-manifolds}, preprint 2016.


\bibitem[NT]{nestandtsygan} R.\ Nest and B.\ Tsygan, \emph{Formal deformations of symplectic manifolds with boundary}, J.\ Reine Angew.\ Math.\ 481, 1996, pp.\ 27--54.

    \bibitem[P]{planas} A.~Planas Bah\'{\i}, \emph{Symplectic surfaces with singularities},  Master Thesis, Universitat Polit\`{e}cnica de Catalunya, 2015.

\bibitem[R]{radko} O.~Radko, \emph{A classification of topologically stable Poisson structures on a compact oriented surface},  J. Symplectic Geom.  \textbf{1} No. 3 (2002), 523--542.

\bibitem[S]{scott} G. \ Scott, \emph{The Geometry of $b^k$ Manifolds},  J. Symplectic Geom. 14 (2016), no. 1, 71--95.

\bibitem[Sw]{swan} R. \ Swan,   \emph{Vector Bundles and Projective Modules}, Transactions of the American Mathematical Society 105, (2), 264--277, (1962).



\bibitem[We]{weinstein} A.\ Weinstein, \emph{The local structure of Poisson manifolds.}, Journal of Differential Geometry 18, no.\ 3, 1983, pp.\ 523--557.

\end{thebibliography}
\end{document}